\documentclass[a4paper,11pt]{article}

\usepackage[margin=2.5cm]{geometry}
\usepackage{enumerate}
\usepackage{hyperref}
\usepackage{amssymb,amsthm,amsmath}
\usepackage{tikz, graphicx, color}

\newcommand{\arxiv}[1]{\href{http://arxiv.org/abs/#1}{\texttt{arXiv:#1}}}

\theoremstyle{plain}
\newtheorem{thm}{Theorem}[section]

\newtheorem{prop}[thm]{Proposition}

\newtheorem{cor}[thm]{Corollary}
\newtheorem{exmp}[thm]{Example}

\theoremstyle{definition}

\numberwithin{equation}{section}

\newcommand\cF{\mathcal{F}}
\newcommand\cG{\mathcal{G}}
\newcommand\cL{\mathcal{L}}
\newcommand\cO{\mathcal{O}}
\newcommand\cP{\mathcal{P}}
\newcommand\cX{\mathcal{X}}
\newcommand\FF{\mathbb{F}}
\newcommand\PP{\mathbb{P}}
\newcommand\ZZ{\mathbb{Z}}
\newcommand\PG{\mathsf{PG}}
\newcommand\GL{\mathsf{GL}}
\renewcommand\le{\leqslant}
\renewcommand\ge{\geqslant}

\newcommand\rank{\mathrm{rank}}
\def\fh{\mathfrak{h}}
\newcommand\Card{\mathrm{Card}}
\DeclareMathOperator\pr{\mathsf{pr}}

\title{The thickness of Schubert cells as incidence structures}

\author{John Bamberg\footnote{The University of Western Australia (john.bamberg@uwa.edu.au)},\quad
Arun Ram\footnote{The University of Melbourne (aram@unimelb.edu.au)},\quad
Jon Xu\footnote{The University of Melbourne (jyz.xu@unimelb.edu.au)}}

\date{\today}

\begin{document}
\maketitle 

\begin{abstract}
This paper explores the possible use of Schubert cells and Schubert varieties in finite geometry,
particularly in regard to the question of whether these objects might be a source of understanding
of ovoids or provide new examples.  The main result provides a characterization of those Schubert cells
for finite Chevalley groups which have the first property (\emph{thinness}) of ovoids.  More importantly, perhaps
this short paper can help to bridge the modern language barrier between finite geometry and representation
theory.  For this purpose, this paper includes very brief surveys 
of the powerful lattice theory point of view from finite geometry and the powerful method
of indexing points of flag varieties by Chevalley generators from representation theory.
\end{abstract}

\section{Introduction}\label{section:intro}

This paper is the result of an effort to create ``interdisciplinary'' communication and collaboration
between the finite geometry community and the representation theory communities in Australia.
The idea was that Chevalley groups could be a bridge between the two languages and the
problems of interest to the two communities.  
Among others, the books of Taylor \cite{Tay92} 
and Buekenhout and Cohen \cite{BC13}
are already existing, useful and important contributions to this dialogue.   Although we have not used the language of buildings
in this paper, the inspiring oeuvre of Tits \cite{Tits74, TitsA, TitsB} is the pinnacle of the powerful 
connections between these different points of view.  See, for example, \cite{PR08} for a brief survey of how these
points of view combine to give insight into the relationship between walks in buildings and representations
of complex algebraic groups and groups over local fields.

We chose to use the finite geometry question
of finding ovoids as a framework for our investigation.  The goal was to shape the language
of algebraic groups and Chevalley groups to provide tools for studying the question. 
The precedent in the work of Tits \cite{Tits61} and Steinberg 
\cite[Example (c) before Theorem 34]{St67} on the Suzuki-Tits ovoid
indicated that this was a fruitful research direction.

To describe further the results and methodology of this paper, let us review the definitions of
ovoids (in finite geometry) and Schubert cells (in representation theory).

\medskip
\emph{Ovoids.}
Let $V$ be a vector space and let $\cP(V)$ be the lattice of subspaces of $V$ with
inclusion $\subseteq$ as the partial order. A \emph{point} is a $1$-dimensional subspace of $V$,
a \emph{line} is a $2$-dimensional subspace and a \emph{hyperplane} is a codimension $1$ subspace of $V$.
Let $\cO$ be a set of points in $\cP(V)$.  A \emph{tangent line to $\cO$}
is a line in $\cP(V)$ that contains exactly one point of $\cO$.
Then \cite[\S1]{Tits62} defines, an \emph{ovoid of $\cP(V)$} as a set $\cO$ of points of $\cP(V)$ such that
\begin{enumerate}
\item[(O1)]  If $\ell$ is a line in $\cP(V)$ then $\ell$ contains 0, 1 or 2 points of $\cO$. \hfill{(thinness)}
\item[(O2)] If $p\in \cO$ then the union of the tangent lines to $\cO$ through 
$p$ is a hyperplane. \hfill{(maximality)}
\end{enumerate}
These two types of conditions, ``thinness'' and ``maximality'', characterize the definitions of
ovoids and ovals and hyperovals lying inside projective spaces, projective planes, polar spaces and
generalized quadrangles that can be found in the finite geometry literature (see, for example, \cite[\S1]{Br00}
and \cite[\S2.1 and \S4.2 and \S4.4]{BW11}).

\medskip
\emph{Schubert cells.}
Let $G(\FF)$ be a Chevalley group over $\FF$ and let $B$ be a Borel subgroup.  The quotient
$G(\FF)/B$ is the (generalized) flag variety. In the case that $G(\FF)= \GL_n(\FF)$,
$G(\FF)/B$ is the set of maximal chains $0\subseteq V_1\subseteq \cdots\subseteq V_{n-1}\subseteq V$
in $\cP(V)$, where $V$ is an $\FF$-vector space of dimension $n$.  The flag varieties are studied with the use of the Bruhat decomposition,
\[
G(\FF) = \bigsqcup_{w\in W} BwB,
\]
and the \textbf{Schubert cells} are 
\[
X_w = BwB
\]
viewed as subsets of the set of cosets $G(\FF)/B$.  In the case of $\GL_n(\FF)/B$ the $X_w$
are collections of maximal chains in $\cP(V)$ and thus, when $\FF=\FF_q$ is a finite field, the
$X_w$ are natural objects in finite geometry.  From the point of view of representation theory,
the closures of the Schubert cells are the Schubert varieties of the projective
variety $G(\overline{\FF})/B$ and this makes them tools in the framework of geometric representation
theory.

In pursuit of the question of what causes the ``thinness'' that distinguishes ovoids we prove 
the following result (Theorem \ref{maintheorem}), which is a computation of 
the ``thickness'' of the incidence structures that come from Schubert cells.

\begin{thm}[Main Theorem]\label{maintheorem} 
Let $G(\FF)$ be a Chevalley group with Weyl group $W$. 
Let $P_i$ and $P_j$ be standard maximal parabolic subgroups of $G(\FF)$ and let $w\in W$.
Let $(X_w)_{ij}$ be the incidence structure associated to the Schubert cell $X_w$ and let $gP_j$
be a line in $(X_w)_{ij}$. Then the number of points in $(X_w)_{ij}$ incident to $gP_j$ is 
\[
q^{\ell(z)},
\]
where $w = uzv$ with $u\in W^j$, $zv \in W_j$, $z\in (W_j)^{i,j}$ and $v \in W_{i,j}$.
\end{thm}

The objects in Theorem \ref{maintheorem} will be defined in forthcoming sections, and in particular,
the incidence structure $(X_w)_{ij}$ will be introduced in Section \ref{section:schubert}.
(It suffices to say here that its `points' are certain left cosets $gP_i$, its `lines' 
are certain left cosets of $hP_j$, and a point and line are incident if the ratio of their canonical 
coset representatives lies in the Borel subgroup of $G(\FF)$).
As an application of this theorem we determine the Schubert cell incidence structures coming from
finite Chevalley groups which have the thinness property; see Corollary \ref{cor:thinschucells}.  

In this paper we first review the background finite geometry of incidence structures and
projective geometries and the notation and framework for working with Chevalley groups and
generalized flag varieties (c.f., Sections \ref{section:lattices} and \ref{section:flagvarieties}). 
In Section \ref{section:schubert}, we define an incidence structure for each 
Schubert cell and pair of maximal parabolic subgroups of the Chevalley group.  This provides a way of 
analyzing the Schubert cell from the viewpoint of finite projective geometry. The main theorem
(Theorem \ref{maintheorem}) is a consquence of Proposition \ref{pullpushdecomp}.


\section{Lattices and incidence structures}\label{section:lattices}

In this section we review the equivalence between subspace lattices of a vector space,
projective lattices and projective incidence structures.  An inspiring modern textbook is \cite{Shu11}.
A classic reference to lattice theory is \cite{Birk48}.  The definition of a modular
lattice is given in \cite[Ch.\ V \S1]{Birk48}.  The equivalence between
projective incidence structures, complemented modular lattices and the subspace
lattice of a vector space over a division ring, which is stated as 
Theorem \ref{subspacesASlattice} below, is proved (even in the infinite dimensional
case) in \cite[Ch.\ VIII, Theorem 15]{Birk48}.
A classic reference to finite geometries is \cite{Dem68}, and the definition of an incidence
structure is given in \cite[\S1.1]{Dem68}. The definition of a projective
incidence structure (often called a projective geometry) is found in 
\cite[Ch.\ VIII \S3]{Birk48}, \cite[\S 3.3]{Cam00} and \cite[p.\ 16]{Tay92}.

\subsection{The subspace lattice $\cP(V)$ of a vector space $V$}

Let $\FF$ be a field or division ring
and let $V$ be a finite dimensional vector space over $\FF$.
The \textbf{subspace lattice} $\cP(V)$ of $V$ is
the set of subspaces of $V$ with partial order given by subspace inclusion.
More generally, one could consider a ring $R$ and a (left) $R$-module $M$ and 
the lattice of (left) $R$-submodules of $M$. At this level of generality, the situation is substantially more involved
and complicated than that of a subspace lattice of a vector space (see \cite{Vel95}).
In the finite geometry literature, a (Desarguesian) \textbf{projective space} is
$\PG(n,q)=\cP(\FF_q^{n+1})$,
where $\FF_q$ is the finite field with $q$ elements.  In the algebraic geometry literature (c.f., \cite[p. 8]{Har77}),
\textbf{projective space} is the quotient 
\[
\PP^n 
= \frac{\FF^{n+1} - \{ (0,\ldots, 0)\}}{\langle (a_0, \ldots, a_n) = (ca_0, \ldots, ca_n)\mid c\in \FF^\times\rangle}.
\]
These terminologies are conflicting and should, therefore, be used with care in the context of this article.

\subsection{Lattices}

A \textbf{lattice} is a partially ordered set $\cP$ that is closed under the operations of meet and join defined by
 $x\vee y=\sup\{x,y\}$ and $x\wedge y=\inf\{x,y\}$, for all $x,y\in cP$.
A \textbf{modular lattice} is a lattice $\cL$ such that for all $x,y,z\in \cL$ such that $x\le z$, we have
\[
x\vee (y\wedge z) = (x\vee y)\wedge z.
\]
Let $\cL$ be a finite lattice with a unique minimal element $0$ and a unique maximal element $1$.
\begin{enumerate}
\item[$\bullet$] An \emph{atom} is $a\in \cL$ such that there does \emph{not} exist $a'\in \cL$ with
$0<a'<a$.
\item[$\bullet$] An \emph{atomic lattice} is a lattice $\cL$ such that every element is a join of atoms.
\item[$\bullet$] A \emph{maximal chain} is a maximal length sequence $0<a_1<a_2<\cdots<a_\ell<1$
in $\cL$.
\item[$\bullet$] A lattice $\cL$ is \emph{ranked} if all maximal chains in $\cL$ have the same length.
\end{enumerate}
Let $\cL$ be a ranked lattice and let $a\in \cL$.  The \textbf{rank} of $a$, written $\rank(a)$, is the integer $i$ for which there exists a maximal
chain 
\[
0<a_1<a_2<\cdots<a_\ell<1
\]
with $a_i=a$.
A \textbf{projective lattice} is an atomic ranked modular lattice such that for all $x,y\in \cL$, we have the \emph{Grassmann identity}:
\[
\rank(x\vee y)+\rank(x\wedge y) = \rank(x)+\rank(y).
\]
Two lattices $\cL$ and $\cL'$ are \emph{isomorphic} if there is an order-preserving bijection from $\cL$ to $\cL'$.
The following theorem provides an equivalence between projective lattices and subspace lattices
of a vector space over a division ring.

\begin{thm}[see {\cite[Chapters V and VI]{HP47}}]\label{subspacesASlattice} \leavevmode
\begin{enumerate}[(a)] 
\item Let $V$ be a finite dimensional vector
space over a division ring.  Then $\cP(V)$ is a projective lattice.
\item If $\cL$ is a projective lattice then there exists a division ring $\FF$ 
and $n\in \ZZ_{>0}$ such that $\cL\cong \cP(\FF^n)$.
\end{enumerate}
\end{thm}

\subsection{Incidence structures}

An \textbf{incidence structure} is a triple $(P,L,I)$ where $P$ and $L$ are sets and $I \subseteq P\times L$.
Let $\pr_1\colon P\times L \to P$ and $\pr_2\colon P\times L\to L$ be the projections onto the first and 
second factors. We have the following interface between geometric language and 
its algebraic formalism:
\begin{itemize}
\item A point $p\in P$ is \emph{contained in a line} $\ell\in L$ if $(p,\ell)\in I$. 
\item A subset
$S\subseteq P$ is \emph{collinear} if there exists $\ell\in L$ such each element $p$ of $S$ is contained in $\ell$. 
\end{itemize}
Often it is convenient to identify $\ell\in L$ with the set of points $\pr_1(\pr_2^{-1}(\ell))$; the points \emph{contained} in the line $\ell$.
A \textbf{projective incidence structure} is an incidence structure $I\subseteq P\times L$ such that
\begin{enumerate}
\item[(a)] If $p_1, p_2\in P$ and $p_1\ne p_2$ then there exists a unique line
$\ell(p_1,p_2)\in L$ containing $p_1$ and $p_2$ (any two points lie on a unique line);
\item[(b)] (Veblen-Young axiom) If $p_1, p_2, p_3\in P$ are not collinear and $\ell$ is a line intersecting
$\ell(p_1,p_3)$ and $\ell(p_2,p_3)$ then $\ell$ also intersects $\ell(p_1,p_2)$;

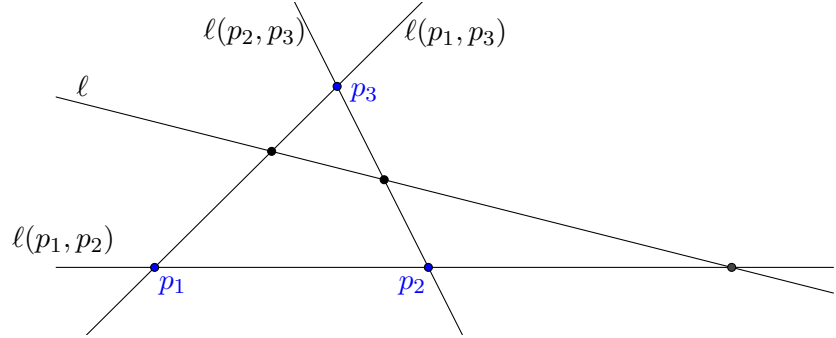
\begin{center}
\definecolor{uuuuuu}{rgb}{0.26666666666666666,0.26666666666666666,0.26666666666666666}
\definecolor{qqqqff}{rgb}{0.,0.,1.}
\definecolor{cqcqcq}{rgb}{0.7529411764705882,0.7529411764705882,0.7529411764705882}
\begin{tikzpicture}[line cap=round,line join=round, x=1.0cm,y=1.0cm,scale=0.6]
\clip(-6,-1.48) rectangle (12.88,5.86);
\draw [domain=-4.16:12.88] plot(\x,{(-0.-0.*\x)/6.});
\draw [domain=-4.16:12.88] plot(\x,{(--8.--4.*\x)/4.});
\draw [domain=-4.16:12.88] plot(\x,{(--16.-4.*\x)/2.});
\draw [domain=-4.16:12.88] plot(\x,{(--6.680660916834094-0.6273965635344791*\x)/2.466961615498298});
\draw [fill=qqqqff] (-2.,0.) circle (2.5pt);
\draw[color=qqqqff] (-1.6,-0.4) node {$p_1$};
\draw [fill=qqqqff] (4.,0.) circle (2.5pt);
\draw[color=qqqqff] (3.65,-0.4) node {$p_2$};
\draw [fill=qqqqff] (2.,4.) circle (2.5pt);
\draw[color=qqqqff] (2.6,3.86) node {$p_3$};
\draw[color=black] (-4,0.6) node {$\ell(p_1,p_2)$};
\draw[color=black] (4.6,5.18) node {$\ell(p_1,p_3)$};
\draw[color=black] (0.2,5.2) node {$\ell(p_2,p_3)$};
\draw [fill=black] (0.5644911108459612,2.564491110845961) circle (2.5pt);
\draw [fill=black] (3.031452726344259,1.937094547311482) circle (2.5pt);
\draw[color=black] (-3.58,4) node {$\ell$};
\draw [fill=uuuuuu] (10.648226823555039,0.) circle (2.5pt);
\end{tikzpicture}
\end{center}

\item[(c)] (thickness condition) Any line contains at least 3 points;
\item[(d)] (dimension $\ge 2$ condition) There exist 3 noncollinear points in $P$;
\item[(e)] (finite dimensionality condition) Any increasing sequence of subspaces has finite length.
\end{enumerate}

Assume that $I\subseteq P\times L$ is an incidence structure such that 
any two points lie on a unique line.
A \textbf{subspace} is a set $S\subseteq P$ such that $S$ contains any line connecting two of its points, i.e.,
if $p_1, p_2\in S$ then $\pr_1(\pr_2^{-1}(\ell(p_1,p_2)))\subseteq S$.
The \textbf{subspace lattice} $\cP(I)$ of $I\subseteq P\times L$ is 
the set of subspaces $S\subseteq P$ partially ordered by inclusion.

The following  ``Veblen-Young Theorem'' provides an equivalence between projective incidence structures
and projective lattices. 

\begin{thm} (see \cite[Chapters V and VI]{HP47})
\item[(a)]  If $\cG$ is a projective incidence structure then $\cP(\cG)$ is a projective lattice.
\item[(b)]  Let $\cP$ be a ranked lattice.  Let 
\begin{align*}
\cP_1 &= \{ p\in \cP \mid \rank(p)=1\}, \\
\cP_2 &= \{ \ell\in \cP \mid \rank(\ell)=2\},
\end{align*}
and let $I$ be the incidence relation inherited from $\cP$; so $(p,\ell)\in \cP_1\times \cP_2$ lies in $I$
if and only if $p\le \ell$ in $\cP(\cG)$.
If $\cP$ is a projective lattice then $(\cP_1,\cP_2,I)$ is a projective incidence structure.
\end{thm}

\section{Flag varieties and Chevalley groups}\label{section:flagvarieties}

In this section we review the formalism and establish our notation for working with (generalized) flag varieties.
A classic reference to Chevalley groups and flag varieties is \cite{St67}.  Good supportive references
are \cite[\S2.1]{Sesh14} and \cite[\S23.3]{FH91}.  
The first step in our review is to identify the flag variety
as the set of maximal chains in the subspace lattice $\cP(V)$.

\subsection{Flag varieties and $\GL_n(\FF)$}\label{subsecflagvarieties}

Let $\FF$ be a field (or division ring) and let $V$ be a finite dimensional $\FF$-vector space.
The \textbf{flag variety} $\cF(V)$ is the set of maximal chains in $\cP(V)$. 
By choosing a basis $\{e_1, \ldots, e_n\}$ in $V$, 
the \textbf{standard flag}
$$F_0 = (0\subseteq \hbox{span}\{e_1\}\subseteq \hbox{span}\{e_1, e_2\}\subseteq
\cdots \subseteq \hbox{span}\{e_1, \ldots, e_n\} = V)$$
has stabilizer the \textbf{Borel subgroup} $B$ consisting of all upper triangle matrices of $\GL_n(\FF)$.
We then obtain a bijection, and an equivalence of group actions (of $\GL_n(\FF)$ on $\GL_n(\FF)/B$ and on $\cF(V)$):
%
\[
\begin{matrix}
\GL_n(\FF)/B &\longrightarrow &\cF(V) \\
gB &\longmapsto &gF_0
\end{matrix}.
\]

A \textbf{parabolic subgroup} of $\GL_n(\FF)$ is the stabilizer of a subspace $W\subseteq V$,
and the \textbf{standard maximal parabolic subgroups} are
\[
P_i = \mathrm{Stab}(\hbox{span}\{e_1, e_2, \ldots, e_i\}),
\]
for $i\in \{1,2, \ldots, n\}$.

Let $E_{ij}$ denote the $n\times n$ matrix with $1$ in the $(i,j)$ entry and $0$ in all other entries.
Let $\fh^*=\ZZ\varepsilon_1+\cdots + \ZZ\varepsilon_n$ be the 
free $\ZZ$-module with basis $\varepsilon_1, \ldots, \varepsilon_n$ and let 
\[
R = \{ \varepsilon_i-\varepsilon_j\mid i,j\in\{1, \ldots, n\}\text{ with }i\ne j\}.
\]
The group $\GL_n(\FF)$ is generated by the \textbf{elementary matrices}
$$x_{\varepsilon_i-\varepsilon_j}(c) = I+cE_{ij},
\quad
s_{\varepsilon_i-\varepsilon_j} = I+E_{ij}+E_{ji}-E_{ii}-E_{jj},
\quad
h_{\lambda^\vee}(d) = \mathrm{diag}(d^{\lambda_1}, \ldots, d^{\lambda_n}),
$$
for $\varepsilon_i-\varepsilon_j\in R$ and $c\in \FF$,
and for $\lambda^\vee = (\lambda_1, \ldots, \lambda_n)\in \ZZ^n$ and $d\in \FF^\times$.
The \textbf{root subgroups} are
\[
\cX_{\varepsilon_i-\varepsilon_j} = \{ x_{\varepsilon_i-\varepsilon_j}(c)\mid c\in \FF\}
\]
and the set of \textbf{positive roots} is
\[
R^+ = \{ \alpha\in R\ | \ \cX_\alpha\subseteq B\}.
\]
The \textbf{simple roots} $\alpha_1, \ldots, \alpha_{n-1}$ are given by
\[
\alpha_i = \varepsilon_i - \varepsilon_{i+1},
\]
and setting $s_i = s_{\alpha_i}$, the \textbf{Weyl group} is
\[
W = \langle s_1, \ldots, s_n\mid s_i^2=1, \ s_is_{i+1}s_i=s_{i+1}s_is_{i+1}\rangle
\]
(which is the symmetric group $S_n$ here).  
The \textbf{Bruhat decomposition} (see \cite[Example (a) after Theorem $4'$]{St67}
or \cite[Theorem 23.59]{FH91} or \cite[\S4.2.4]{Sesh14}) 
is
$$\GL_n(\FF) = \bigsqcup_{w\in W} BwB.$$

\subsection{Chevalley groups and generalized flag varieties $G(\FF)/B$}\label{ChevalleyGroups}

In the same way that $\GL_n(\FF)$ is generated by elementary matrices, a Chevalley group $G(\FF)$
is generated by Chevalley generators $x_\alpha(c)$, $h_{\lambda^\vee}(d)$, 
for $\alpha,\lambda \in R$, $c\in \FF$, $d\in \FF^\times$, 
which satisfy specified relations \cite[Relations (R), Chapter 3, page 23]{St67}.  
The set $R$ of \textbf{roots} is a labeling
set for the \textbf{root subgroups}
\begin{equation}
\cX_\alpha = \{ x_\alpha(c)\mid c\in \FF\}
\quad\hbox{for $\alpha\in R$.}
\end{equation}

The set $R$ is endowed with a chosen decomposition into positive and negative roots
\begin{equation*}
R = R^+\sqcup (-R^+), \qquad\hbox{where\quad $-R^+ = \{ -\alpha\mid \alpha\in R^+\}$}.
\end{equation*}
Defining
\[
U = \langle \cX_\alpha\mid \alpha\in R^+\rangle,
\quad T= \langle h_{\lambda^\vee}(d)\mid \lambda^\vee\in \fh_\ZZ,\ d\in \FF^\times\},
\quad\hbox{and}\quad B=UT,
\]
we call $G(\FF)/B$ the \textbf{generalized flag variety}.
The \textbf{simple roots} $\alpha_1, \ldots, \alpha_n$ provide a minimal set of root subgroup
generators for 
\[
U = \langle \cX_{\alpha_1}, \ldots, \cX_{\alpha_n}\rangle.
\]
The \textbf{standard maximal parabolic subgroups} are
\[
P_i = \langle \cX_{-\alpha_1}, \ldots, \cX_{-\alpha_{i-1}},\cX_{-\alpha_{i+1}},\ldots, \cX_{-\alpha_n}, B\rangle,
\qquad
\hbox{for $i\in \{1, \ldots, n\}$.}
\]

\subsubsection{Labeling the points of the flag variety}

Letting $N = \langle n_\alpha\mid \alpha\in R\rangle$, the \emph{Weyl group} is $W = N/T$.
For $i\in \{1, \ldots, n\}$, define
\[
x_i(c) = x_{\alpha_i}(c), \qquad
n_i = x_{\alpha_i}(1)x_{-\alpha_i}(-1)x_{\alpha_i}(1),
\qquad\hbox{and}\quad s_i = n_i T.
\]
The Weyl group $W$ has a Coxeter presentation with generators
$s_1, \ldots, s_n$ and relations
$s_i^2 = 1$ and $(s_is_j)^{m_{ij}}=1$,
where $m_{ij}$ is the order of $s_is_j$ in $W$.
A \textbf{reduced decomposition} for an element $w\in W$ is an expression $w=s_{i_1}\cdots s_{i_\ell}$ with $\ell$ minimal.
The following provides an explicit indexing of the points of the flag variety.


\begin{prop}[{\cite[Theorem $4'$, Theorem 15 and Lemma 43(a)]{St67}
see also \cite[(7.3)]{PRS}}]\label{GBreps}
 For each $w \in W$ fix a reduced decomposition $w = s_{i_1}\cdots s_{i_\ell}$.  Then
$$G(\FF)/B = \bigsqcup_{w \in W} BwB
\qquad\hbox{with}\quad
BwB = \{ x_{i_1}(c_1)n_{i_1}^{-1}\cdots x_{i_\ell}(c_\ell)n_{i_\ell}^{-1}B\mid c_1, \ldots, c_{\ell} \in \FF\},$$
and 
$\{ x_{i_1}(c_1)n_{i_1}^{-1}\cdots x_{i_{\ell}}(c_{\ell})n_{i_{\ell}}^{-1}\mid c_1, \ldots, c_{\ell}\in \FF\}$ is a complete set
of representatives of the cosets of $B$ in $BwB$.
\end{prop}

\section{Thickness in Schubert cells}\label{section:schubert}

Keeping the notation of Section \ref{ChevalleyGroups},
let $G(\FF)$ be a Chevalley group and let $P_i$ and $P_j$ be standard maximal parabolic subgroups of $G(\FF)$.  
Let $w\in W$. 
Define maps $p_i^w$ and $p_j^w$ as follows:
\[
\begin{matrix}
p^w_i\colon &BwB &\to &G/{P_i} \\
&gB &\mapsto &gP_i
\end{matrix}
\qquad\hbox{and}\qquad 
\begin{matrix}
p^w_j\colon &BwB &\to &G/P_j \\
&gB &\mapsto &gP_j\ .
\end{matrix}
\] 
Let $(X_w)_{ij}$ be the following incidence structure:
%
\begin{enumerate}
\item[(a)] a \emph{point} in $(X_w)_{ij}$ is an element $gP_i$ of the image of $p_i^w$,
\item[(b)] a \emph{line} in $(X_w)_{ij}$ is an element $hP_j$ of the image of $p_j^w$, and 
\item[(c)] a point $gP_i$ is \emph{incident} to a line $hP_j$ if there exists $kB\in BwB$ such that
$p_i^w(kB) = gP_i$ and $p_j^w(kB) = hP_j$.
\end{enumerate}
Alternatively, it is not difficult to see that the incidence relation above can be simplified by stipulating
that $gh^{-1}\in B$ instead.

Let
\[
R_i^+ = \{ \alpha\in R^+\mid \cX_{-\alpha}\in P_i\}, \quad R_j^+ = \{ \alpha\in R^+\mid \cX_{-\alpha}\in P_j\}, \quad
R_{\{i,j\}}^+ = R_i^+\cap R_j^+,
\]
and let
\[
W_i = \langle s_\alpha\mid \alpha\in R_i^+\rangle, W_j = \langle s_\alpha\mid \alpha\in R_j^+\rangle, W_{\{i,j\}} = W_i\cap W_j.
\]
For $z\in W$ the \emph{inversion set of $z$} is
\[
R(z) := \{\alpha\in R^+\mid \cX_{z\alpha}\not\in B\},
\]
and $\ell(z) := \Card(R(z))$ is the length of a reduced decomposition of $z$ (in this definition 
$\cX_{z\alpha} = z\cX_\alpha z^{-1}$).  
Let $W^j$ be the set of minimal length coset representatives of $W_j$ in $W$,
and let $W_j^{\{i,j\}}$ be the set of minimal length coset representatives of $W_i\cap W_j$ in $W_j$.
So
\begin{align*}
W^j 
&= \{ z\in W\mid R(z)\cap R_j^+ = \varnothing\}, \\
(W_j)^{\{i,j\}} 
&=\{ z\in W\mid R(z)\subseteq R_j^+\text{ and }R(z)\cap R_{\{i,j\}} = \varnothing\}.
\end{align*}

The following proposition is a slight generalization of Propositon \ref{GBreps}.

\begin{prop}\label{GPireps} 
For each $u\in W^j$ fix a reduced decomposition $u = s_{i_1}\cdots s_{i_k}$.  Then
\[
G/P_j = \bigsqcup_{u \in W^j} BuP_j
\qquad\hbox{with}\quad
BuP_j = \{ x_{i_1}(c_1)n_{i_1}^{-1}\cdots x_{i_k}(c_k)n_{i_k}^{-1}P_j\mid c_1, \ldots, c_k\in \FF\},
\]
and 
$\{ x_{i_1}(c_1)n_{i_1}^{-1}\cdots x_{i_k}(c_k)n_{i_k}^{-1}\mid c_1, \ldots, c_k\in \FF\}$ is a set
of representatives of the cosets of $P_j$ in $BuP_j$.
\end{prop}
\begin{proof}
If $w\in W$ then there are unique $u\in W^j$ and $y\in W_j$ such that $w=uy$ (see \cite[Ch.\ 4 \S1 Exercise 3]{Bou}).
If $u=s_{i_1}\cdots s_{i_k}$ and $y=s_{i_{k+1}}\cdots s_{i_\ell}$ are reduced decomposition
then $w=s_{i_1}\cdots s_{i_k}s_{i_{k+1}}\cdots s_{i_\ell}$ is reduced.  If 
$gB = x_{i_1}(c_1)n_{i_1}^{-1}\cdots x_{i_\ell}(c_\ell)n_{i_\ell}^{-1}B \in BwB$
then $gP_i = x_{i_1}(c_1)n_{i_1}^{-1}\cdots x_{i_k}(c_k)n_{i_k}^{-1}P_j$ since
every factor of the product 
$x_{i_{k+1}}(c_{k+1})n_{i_{k+1}}^{-1}\cdots x_{i_\ell}(c_\ell)n_{i_\ell}^{-1}$ is an
element of $P_j$.
\end{proof}

%

Let $p_i\colon G/B\to G/P_i$ and $p_j\colon G/B\to G/P_j$ be the natural projection maps (e.g., 
$p_i(gB)=gP_i$ for all $g\in G$).
Each $y\in W_j$ has a unique expression $y=zv$ with $z\in (W_j)^{\{i,j\}}$ and
$v\in W_{\{i,j\}}$.  
For each $y\in W_j$, fix a reduced decomposition 
\begin{equation*}
y = s_{k_1}\cdots s_{k_r} s_{\ell_1}\cdots s_{\ell_t},
\quad\hbox{with $z=s_{k_1}\cdots s_{k_r}\in (W_j)^{\{i,j\}}$ and 
$v=s_{\ell_1}\cdots s_{\ell_t}(W_j)_{\{i,j\}}$.}
\end{equation*}
With
\begin{equation}
U_y = \{ x_{k_1}(d_1)n_{k_1}^{-1} \cdots x_{k_r}(d_r)n_{k_r}^{-1}
x_{\ell_1}(e_1)n_{\ell_1}^{-1}\cdots x_{\ell_t}(e_t)n_{\ell_t}^{-1}
\mid d_1, \ldots, d_r, e_1, \ldots, e_t\in \FF\},
\label{Uydefn}
\end{equation}
we have
\begin{equation}
P_j = \bigsqcup_{y\in W_j} U_yB
\qquad\hbox{and}\qquad
p_j^{-1}(gP_j) = \bigsqcup_{y\in W_j} gU_y B.
\label{fiberreps}
\end{equation}
With this notation in hand, we can now state the following proposition that determines the structure of each $p_i(p_j^{-1}(gP_j)$.

\begin{prop}\label{pullpushdecomp} Let $gP_j\in G/P_j$.
With notation as above, 
the map $\Phi$ from $p_i(p_j^{-1}(gP_j))$ to $\bigsqcup_{y\in W_j} \FF^{\ell(z)}$ defined by
\[
\Phi\left(gx_{k_1}(d_1)n_{k_1}^{-1}\cdots x_{k_r}(d_r)n_{k_r}^{-1}P_i\right) := (d_1, \ldots, d_r)
\]
is a bijection.
%
\end{prop}
\begin{proof}
By \eqref{fiberreps}, the set $p_j^{-1}(gP_j)$ is a disjoint union of 
the sets $gU_y$ for $y\in W^j$.  By \eqref{Uydefn}, 
an element of $gU_yB$ is of the form $gx_{k_1}(d_1)n_{k_1}^{-1}\cdots x_{k_r}(d_r)n_{k_r}^{-1}
x_{\ell_1}(e_1)n_{\ell_1}^{-1}\cdots x_{\ell_t}(e_t)n_{\ell_t}^{-1}
B$ and then
\begin{align*}
p_i(gx_{k_1}&(d_1)n_{k_1}^{-1}\cdots x_{k_r}(d_r)n_{k_r}^{-1}
x_{\ell_1}(e_1)n_{\ell_1}^{-1}\cdots x_{\ell_t}(e_t)n_{\ell_t}^{-1}
B) \\
&= gx_{k_1}(d_1)n_{k_1}^{-1}\cdots x_{k_r}(d_r)n_{k_r}^{-1}
x_{\ell_1}(e_1)n_{\ell_1}^{-1}\cdots x_{\ell_t}(e_t)n_{\ell_t}^{-1}
P_i  \\
&= gx_{k_1}(d_1)n_{k_1}^{-1}\cdots x_{k_r}(d_r)n_{k_r}^{-1}
P_i .
\end{align*}
Thus each element of $p_i(p_j^{-1}(gP_j))$ can be written in the form 
$gx_{k_1}(d_1)n_{k_1}^{-1}\cdots x_{k_r}(d_r)n_{k_r}^{-1}P_i$.  

Now let $z_1, z_2\in (W_j)^{\{i,j\}}$ with chosen reduced decompositions
$$z_1 = s_{k_1}\cdots s_{k_r}
\qquad\hbox{and}\qquad
z_2 = s_{k_1'}\cdots s_{k_m'}.
$$
Assume
$$
gx_{k_1}(d_1)n_{k_1}^{-1}\cdots x_{k_r}(d_r)n_{k_r}^{-1} P_i
=gx_{k_1'}(d_1')n_{k_1'}^{-1}\cdots x_{k_m'}(d_m')n_{k_m'}^{-1}P_i.
$$
Then
$$
x_{k_1}(d_1)n_{k_1}^{-1}\cdots x_{k_r}(d_r)n_{k_r}^{-1} P_i
=x_{k_1'}(d'_1)n_{k_1'}^{-1}\cdots x_{k_m'}(d_m')n_{k_m'}^{-1}P_i.
$$

Since $z_1\in (W_j)^{\{i,j\}}$ and $R^+_{\{i,j\}}\subseteq R_i^+$,
we have $R(z)\cap R_i^+ \subseteq R(z) \cap R_{\{i,j\}}^+ = \varnothing$, giving that $z_1\in W^i$. 
Similarly $z_2\in W^i$.
Since
$x_{k_1}(d_1)n_{k_1}^{-1}\cdots x_{k_r}(d_r)n_{k_r}^{-1} P_i
=x_{k_1'}(d'_1)n_{k_1'}^{-1}\cdots x_{k_m'}(d'_m)n_{k_m'}^{-1}P_i$,
we have $z_1W_i = z_2W_i$. Since $z_1$ and $z_2$ are minimal length coset representatives
of the same coset in $W/W_i$, and since such coset representatives are unique (see \cite[Ch.\ 4 \S1 Exercise 3]{Bou}), we find that 
\[
z_1=z_2.\]
Since the reduced decompositions of 
elements of $(W_j)^{\{i,j\}}$ were fixed,
\[
(k_1, \ldots, k_r) = (k_1', \ldots, k_m').
\]
By Proposition \ref{GPireps}, since $x_{k_1}(d_1)n_{k_1}^{-1}\cdots x_{k_r}(d_r)n_{k_r}^{-1} P_i
=x_{k_1}(d'_1)n_{k_1}^{-1}\cdots x_{k_r}(d'_r)n_{k_r}^{-1}P_i$, we have
\[
(d_1, \ldots, d_r) = (d'_1, \ldots, d'_r).
\]
Thus each element of $p_i(p_j^{-1}(gP_j))$ has a \emph{unique} expression as
$gx_{k_1}(d_1)n_{k_1}^{-1}\cdots x_{k_r}(d_r)n_{j_k}^{-1}P_i$.  
\end{proof}

\begin{proof}[Proof of Theorem \ref{maintheorem}]  Let $w\in W$ and let $gP_j$ be in the image of $p^w_j\colon BwB\to G/P_j$.
The decomposition $w=uy=uzv$ is unique (see \cite[Ch.\ 4 \S1, Exercise 3]{Bou}).
Thus $z$ is determined. Hence by Proposition \ref{pullpushdecomp}, the set
\begin{align*}
p_i^w(p_j^w)^{-1}(gP_j) = p_i^w(X_w) \cap p_ip_j^{-1}(gP_j)
\end{align*}
has $q^{\ell(z)}$ elements.
\end{proof}

\begin{exmp}
Take $G=G(\FF)= \GL_4(\FF)$ and the notation given in Section \ref{subsecflagvarieties}.
Let $i=1$ and $j=2$.  Then
$$
W=S_4,\quad W_1 = S_1\times S_3, \quad W_2 = S_2\times S_2,
\quad
W_{1,2} = S_1\times S_1\times S_2,$$
and 
\[
W^1 = \{ 1, s_1, s_2s_1, s_3s_2s_1\},\quad
W^2=\{ 1, s_2, s_1s_2, s_3s_2, s_1s_3s_2, s_2s_1s_3s_2\}
\quad\text{and}\quad
(W_2)^{1,2} = \{1, s_1\}.
\]
Let  $w=uzy = (s_1s_3s_2)(s_1)(s_3)$.
Consider the incidence structure $(X_w)_{12}$ and 
\[
g = x_1(c_1)n_1^{-1}x_3(c_2)n_3^{-1}x_2(c_3)n_2^{-1}.
\]
Then
\begin{align*}
p_1(p_2^{-1}(gP_2))&=p_1(p_2^{-1}(x_1(c_1)n_1^{-1}x_3(c_2)n_3^{-1}x_2(c_3)n_2^{-1}P_2)) \\
&= \left\{ x_1(c_1)n_1^{-1}x_3(c_2)n_3^{-1}x_2(c_3)n_2^{-1}
x_1(d_1)n_1^{-1}P_1\mid d_1\in \FF\right\}.
\end{align*}
This illustrates that $p_1(p_2^{-1}(gP_2))\cong \FF$ even though the elements of $p_1(p_2^{-1}(gP_2))$
as displayed
are not the ``favourite'' coset representatives of the cosets
in $G/P_1$ given by Proposition \ref{GPireps}.  This provides a conceptual explanation of why 
Proposition \ref{pullpushdecomp} (and Theorem \ref{maintheorem}) are nontrivial. One needs to find the 
right coordinatization to succeed in displaying $p_1(p_2^{-1}(gP_2))$ naturally as an affine space.
\end{exmp}

Recall from the introduction that the first of the defining conditions for an ovoid $\cO$ in $\cP(V)$ is
`thinness' (O1): any $\ell$ of $\cP(V)$ contains at most two points of $\cO$.
Using Theorem \ref{maintheorem} to determine the Schubert incidence structures that are `thin' produces
the following result.

\begin{cor} \label{cor:thinschucells} Let $G(\FF_q)$ be a Chevalley group over a finite field $\FF_q$. Then the Schubert incidence structures
$(X_w)_{ij}$ such that there are at most two points incident with each line correspond to triples $(w,i,j)$ such that
\[
\begin{cases}
w\in W^jW_{i,j}, &\hbox{if $q>2$,} \\
w\in W^jW_{i,j}\cup W^js_iW_{i,j}, &\hbox{if $q=2$.}
\end{cases}
\]
\end{cor}
\begin{proof}
Assume $w=uzy$ with $u\in W^j$, $z\in (W_j)^{i,j}$, $y\in W_{i,j}$.  Then
$\ell(z) = 0$ only when $z=1$ and $\ell(z)=1$, and this occurs only when $z = s_i$.
\end{proof}





\end{document}